\documentclass[12pt, reqno]{amsart}
\usepackage{}
\usepackage{amsfonts}
\usepackage{amsmath, amsthm, amscd, amsfonts, amssymb, graphicx, color}
\usepackage[bookmarksnumbered, colorlinks, plainpages]{hyperref}

\newtheorem{theorem}{Theorem}[section]
\newtheorem{lemma}[theorem]{Lemma}
\newtheorem{proposition}[theorem]{Proposition}
\newtheorem{corollary}[theorem]{Corollary}
\theoremstyle{definition}
\newtheorem{definition}[theorem]{Definition}

\theoremstyle{remark}
\newtheorem{remark}[theorem]{Remark}
\numberwithin{equation}{section}

\begin{document}
\setcounter{page}{1}

\title[Coarse Quotient Mappings between Metric Spaces]{Coarse Quotient Mappings between Metric Spaces}

\author[Sheng Zhang]{Sheng Zhang}

\address{Department of Mathematics, Texas A\&M University, College Station, TX 77843-3368, USA}
\email{\textcolor[rgb]{0.00,0.00,0.84}{z1986s@math.tamu.edu}}

%\dedicatory{This paper is dedicated to Professor ABCD}


%\keywords{Convexity, stability, functional equation, Hahn--Banach theorem.}

%\date{Received: xxxxxx; Revised: yyyyyy; Accepted: zzzzzz.
%\newline \indent $^{*}$ Corresponding author}

\begin{abstract}
We give a definition of coarse quotient mapping and show that
several results for uniform quotient mapping also hold in the coarse
setting. In particular, we prove that any Banach space that is a
coarse quotient of $L_p\equiv L_p[0,1]$, $1<p<\infty$, is isomorphic
to a linear quotient of $L_p$. It is also proved that $\ell_q$ is
not a coarse quotient of $\ell_p$ for $1<p<q<\infty$ using
Rolewicz's property ($\beta$).
\end{abstract} \maketitle

\section{Introduction}

In the linear theory of Banach spaces, a surjective bounded linear
operator between Banach spaces is called a linear quotient mapping.
For linear quotient mappings, the Open Mapping Theorem guarantees
that the image of the closed unit ball contains the origin as an
interior point. In view of this property, the notion of quotient
mapping was generalized to the nonlinear setting by Bates, Johnson,
Lindenstrauss, Preiss and Schechtman in \cite{Bates}, as follows:

\begin{definition}
A mapping $f:X\rightarrow Y$ between two metric spaces $X$ and $Y$
is called co-uniformly continuous if for every $\epsilon>0$, there
exists $\delta=\delta(\epsilon)>0$ so that for every $x\in X$,
$$f(B(x,\epsilon))\supset B(f(x),\delta).$$ If $\delta$ can be
chosen larger than $\epsilon/C$ for some $C>0$ independent of
$\epsilon$, then $f$ is said to be co-Lipschitz. Here and throughout
this article $B(x,\epsilon)$ denotes the closed ball with center $x$
and radius $\epsilon$.

A mapping $f:X\rightarrow Y$ that is both uniformly continuous and
co-uniformly continuous (resp. Lipschitz and co-Lipschitz) is called
a uniform quotient (resp. Lipschitz quotient) mapping. If in
addition $f$ is surjective, then $Y$ is called a uniform quotient
(resp. Lipschitz quotient) of $X$.
\end{definition}

In addition to uniform continuity and Lipschitz continuity, coarse
continuity is another important notion for nonlinear mappings.
However, so far there is no satisfactory definition for quotient
mapping in the coarse category. In the present article we introduce
such a definition, which is reasonable in the sense that coarse
quotient mappings have properties analogous to those of uniform
quotient mappings.

This paper is organized as follows: in section 2 we present some
ingredients we need from coarse geometry and then give the definition
of coarse quotient mapping. Section 3 is devoted to ultraproduct
techniques for coarse quotient mappings, which allow us to give an
isomorphic characterization of coarse quotients of $L_p$
($1<p<\infty$). In section 4 we apply Rolewicz's property ($\beta$)
to study coarse quotient mappings and prove that $\ell_q$ is not a
coarse quotient of $\ell_p$ for $1<p<q<\infty$.

\section{Definition of coarse quotient mapping}

We start this section by listing some definitions and facts from
coarse geometry.

\begin{definition}\label{CC}
A mapping $f:X\rightarrow Y$ between two metric spaces $X$ and $Y$
is called coarsely continuous if for every $R>0$ there exists
$S=S(R)>0$ such that $d(x,y)<R\Rightarrow d(f(x),f(y))<S$.

A mapping $f:X\rightarrow Y$ between two metric spaces $X$ and $Y$
is said to be a coarse equivalence if $f$ is coarsely continuous and
there exists another coarsely continuous mapping $g:Y\rightarrow X$
such that $\sup\{d(g\circ f(x),x): x\in X\}<\infty$ and
$\sup\{d(f\circ g(y),y): y\in Y\}<\infty$. In this case we say that $X$ is
coarsely equivalent to $Y$. If $X$ is coarsely equivalent to a subset of $Y$
we say that $X$ coarsely embeds into $Y$.
\end{definition}

It is well known that $f$ is coarsely continuous if and only if
$\omega_f(t)<\infty$ for all $0<t<\infty$, where $\omega_f$ is the
modulus of continuity of $f$, defined as
$$\omega_f(t)=\sup\{d(f(x),f(y)): d(x,y)\leq t\}.$$ The following
proposition says that a coarsely continuous mapping between metric
spaces is Lipschitz for large distances provided its domain space is
metrically convex. The proof is similar to the one for uniformly
continuous mapping, which can be found in \cite{BL}. Recall that a metric space
$X$ is called metrically convex if for every $x_0, x_1\in X$ and for
every $0<t<1$, there is a point $x_t\in X$ such that $d(x_0,
x_t)=td(x_0, x_1)$ and $d(x_1, x_t)=(1-t)d(x_0, x_1)$.

\begin{proposition}\label{LipLarge}
Let $f:X\rightarrow Y$ be a coarsely continuous mapping between two
metric spaces $X$ and $Y$, and assume that $X$ is metrically convex.
Then for every $d>0$, there exists $l=l(d)>0$ such that
$d(f(x),f(y))\leq l\cdot d(x,y)$ whenever $d(x,y)\geq d$. Or,
equivalently, for every $\epsilon>0$, there exists $L=L(\epsilon)>0$
such that $f(B(x,r))\subset B(f(x), Lr)$
for all $r\geq\epsilon$ and all $x\in X$.
\end{proposition}

Now we are ready to give the definition of coarse quotient mapping.
Given $\varepsilon\geq0$, for a subset $A$ of a metric space $X$ we
denote by $A^{\varepsilon}:=\{x\in X: d(x,y)\leq\varepsilon~
\text{for some}~y\in A\}$ its $\varepsilon$-neighborhood.

\begin{definition}\label{CQ}
Let $K\geq0$ be a constant. A mapping $f:X\rightarrow Y$ between
two metric spaces $X$ and $Y$ is called co-coarsely continuous
with constant $K$ if for every $\epsilon>0$ there exists
$\delta=\delta(\epsilon)>0$ so that for every $x\in X$,
\begin{align}\label{CQinclu}
B(f(x),\epsilon)\subset f(B(x,\delta))^K.
\end{align}
If in addition $f$ is coarsely continuous, then $f$ is called a coarse quotient mapping
with constant $K$ and $Y$ is said to be a coarse quotient of $X$.
\end{definition}

\begin{remark}
\eqref{CQinclu} immediately implies that $f(X)$ is $K$-dense in $Y$, i.e., $Y=f(X)^K$.
\end{remark}

It is easy to check that the composition of two coarse quotient
mappings is still a coarse quotient mapping. A Lipschitz quotient
mapping must be a coarse quotient mapping, but the converse is not
true. Indeed, the inclusion mapping from $\mathbb{Z}$ into
$\mathbb{R}$ is an example of coarse quotient mapping which fails
to be a uniform quotient mapping. Moreover, the definition is
justified by the following proposition.

\begin{proposition}
If a mapping $f:X\rightarrow Y$ between two metric spaces $X$ and
$Y$ is a coarse equivalence, then $Y$ is a coarse quotient of $X$.
\end{proposition}

\begin{proof}
By the definition of coarse equivalence, there exists a coarsely
continuous mapping $g:Y\rightarrow X$ such that
\begin{align}
&\sup\{d(g\circ f(x),x): x\in X\}\leq M_1,\nonumber\\
&\sup\{d(f\circ g(y),y): y\in Y\}\leq M_2,\nonumber
\end{align}
where $M_1$ and $M_2$ are nonnegative constants. To see that $f$ is a coarse quotient mapping,
we claim that in Definition \ref{CQ} the constant $K$ can be chosen as $M_2$, and for
every $\epsilon>0$, $\delta=\delta(\epsilon)$ can be chosen as
$\omega_g(\epsilon)+M_1$.

Indeed, for every $x\in X$ and every $y\in B(f(x),\epsilon)$, one has
$d(g\circ f(x),g(y))\leq\omega_g(\epsilon)$, thus $d(g(y),x)\leq\omega_g(\epsilon)+M_1$
by the triangle inequality, and hence $z:=g(y)$ satisfies
$d(y,f(z))=d(y,f\circ g(y))\leq M_2$.
\end{proof}

In general, coarse quotient mappings are not necessarily surjective:
\eqref{CQinclu} only implies that $Y=f(X)^K$.
However the next lemma shows that in the Banach space setting,
one can always redefine a coarse quotient mapping
to have constant $K=0$. This striking fact was pointed out by W. B. Johnson. The underlying
idea is to use transfinite induction based on the oberservation that if a Banach space $Y$
is a coarse quotient of a Banach space $X$, then $\text{card}(X)\geq\text{card}(Y)$. Here
$\text{card}(X)$ is the cardinality of $X$. Indeed, if $f:X\rightarrow Y$ is a coarse quotient mapping,
then for any dense subset $S$ of $X$, $f(S)$ is $D$-dense in $Y$ for some $D\geq0$.
This implies that $\text{dens}(X)\geq\text{dens}(Y)$, where $\text{dens}(X)$ is the density character of $X$.
Therefore $\text{card}(X)\geq\text{card}(Y)$ (see, e.g., \cite{MO}).

Let $\delta>0$. We say that a subset $N$ is a $\delta$-net of a metric space $X$
if it is $\delta$-dense in $X$ and $\delta$-separated, i.e., $d(u,v)\geq\delta$ for all
$u,v\in N$. The existence of nets is guaranteed by Zorn's Lemma.

\begin{lemma}\label{Johnson}
Let $X$ and $Y$ be Banach spaces. Assume that
$Y$ is a coarse quotient of $X$. Then there exists a coarse quotient
mapping with constant 0 from $X$ onto $Y$.
\end{lemma}

\begin{proof}
Let $f:X\rightarrow Y$ be a coarse quotient mapping and
$N$ be a 1-net of $X$. Since $N$ is coarsely equivalent to $X$, $f|_N:N\rightarrow Y$ is still
a coarse quotient mapping, say, with constant $K\geq0$. Thus
for every $\epsilon>0$ there exists $\delta=\delta(\epsilon)>0$ so that for every $x\in N$,
\begin{align}\label{CQN}
B(f(x),\epsilon)\subset f(B(x,\delta)\cap N)^K.
\end{align}
Consider the set
$$\Gamma=\{(x,\epsilon,y):x\in N,\epsilon\in\mathbb{Q},y\in B(f(x),\epsilon)\},$$
one has
\begin{align}\label{card}
\kappa:=\text{card}(\Gamma)\leq\max\{\text{card}(N)
,\text{card}(Y)\}\leq\text{card}(X).
\end{align}
Fix a well-ordering $\preceq$ of $\Gamma$ of order-type $\kappa$
(i.e., each element has strictly fewer than $\kappa$ predecessors),
we will define by transfinite induction on $\alpha\in\Gamma$
new mappings $g_\alpha\subset X\times Y$ such that
$f|_N\subset g_\alpha\subset g_\beta$ for all $\alpha\preceq\beta$ in $\Gamma$.
The desired mapping $g$ will be $\bigcup_{\alpha\in\Gamma}g_\alpha$, whose
domain $\text{dom}g:=\widetilde{N}$ contains $N$ as a subset.

Suppose $g_\beta$ has been defined for all $\beta\prec\alpha$.
By \eqref{CQN} there exist $\delta_\alpha=\delta(\epsilon_\alpha)>0$
and $u_\alpha\in B(x_\alpha,\delta_\alpha)\cap N$
so that $\|y_\alpha-f(u_\alpha)\|\leq K$.
Note that $\delta_\alpha\geq\|x_\alpha-u_\alpha\|\geq1$,
so replacing $\delta_\alpha$ by $2\delta_\alpha$
we get $B(u_\alpha,1/2)\subset B(x_\alpha,2\delta_\alpha)$.
In view of \eqref{card}, we can pick
$$v_\alpha\in B(u_\alpha,\frac{1}{2})\backslash\left(\bigcup_
{\beta\prec\alpha}\text{dom}g_\beta\cup N\right)$$
and define $g_\alpha$ by
$$g_\alpha=\bigcup_{\beta\prec\alpha}g_\beta\cup(v_\alpha,y_\alpha).$$

Clearly $g$ is surjective, and it follows from the choice of
$v_\alpha$ and the triangle inequality that
$$\|f(v_\alpha)-g(v_\alpha)\|\leq\omega_f(\frac{1}{2})+K,$$
so $g$ is coarsely continuous.
Moreover, $g$ satisfies the local surjectivity condition at points of $N$, i.e.,
for every $\epsilon>0$ there exists $\delta=\delta(\epsilon)>0$ so that for every $x\in N$,
$$B(g(x),\epsilon)\subset g(B(x,\delta)\cap\widetilde{N}).$$

The last step is to extend $g$ to all of $X$. Consider the selection mapping
$p:X\rightarrow\widetilde{N}$ defined as $p(x)=x$ for $x\in\widetilde{N}$ and $p(x)=u_x$
for $x\in X\backslash\widetilde{N}$, where $u_x$ is any point in $N$ within distance 1
from $x$. Then one can easily check that the composition
$g\circ p$ is a coarse quotient mapping with constant 0 from $X$ onto $Y$.
\end{proof}

As an application of Lemma \ref{Johnson}, we give the following corollary.

\begin{corollary}\label{banachcontain}
Let $X$, $Y$ and $Z$ be Banach spaces. Assume that $Y$ is a coarse quotient of $X$
and there is no coarse quotient mapping from any subset of $X$ to $Z$ that is Lipschitz
for large distances. Then $Z$ does not isomorphically embed into $Y$.
\end{corollary}

\begin{proof}
Assume that $Y$ contains a subspace $\widetilde{Z}$ isomorphic to $Z$ and
the isomorphism is given by $T: \widetilde{Z}\rightarrow Z$. Let $f: X\rightarrow Y$
be a coarse quotient mapping with constant $K$. In view of Lemma \ref{Johnson} we may assume that $K=0$.
Consider the subset $S:=f^{-1}(\widetilde{Z})$ of $X$. Clearly
$f|_S: f^{-1}(\widetilde{Z})\rightarrow\widetilde{Z}$ is still a coarse quotient mapping with
constant 0, and as a restriction mapping it inherits the Lipschitz for large distances property.
Thus the composition $T\circ f|_S$ is a coarse quotient mapping from
$S$ to $Z$ that is Lipschitz for large distances, a contradiction.
\end{proof}

Corollary \ref{banachcontain} can be generalized to the metric space setting,
but the proof is more complicated since Lemma \ref{Johnson} does not apply to
general metric spaces.

\begin{theorem}\label{metriccontain}
Let $X$, $Y$ and $Z$ be metric spaces. Assume that $Y$ is a coarse quotient of $X$
and there is no coarse quotient mapping from any subset of $X$ to $Z$. Then $Z$
does not coarsely embed into $Y$.
\end{theorem}

\begin{proof}
Assume that $\widetilde{Z}$ is a subset of $Y$ that is coarsely equivalent to $Z$ and
the coarse equivalence is given by $T: \widetilde{Z}\rightarrow Z$. Let $f: X\rightarrow Y$
be a coarse quotient mapping with constant $K$.
Consider the subset $S:=f^{-1}(\widetilde{Z}^K)$ of $X$. We claim that the restriction mapping
$f|_S: f^{-1}(\widetilde{Z}^K)\rightarrow\widetilde{Z}^K$ is a coarse quotient mapping.

Indeed, let $x\in S$ and $\epsilon>0$.
Since $f: X\rightarrow Y$ is co-coarsely continuous, there exists $\delta=\delta(\epsilon)>0$ so that
\eqref{CQinclu} holds. Thus for every $y\in B(f(x),\epsilon)\cap\widetilde{Z}^K$
there exists $u\in B(x,\delta)$ such that $d(y, f(u))\leq K$. On the
other hand, $y\in\widetilde{Z}^K$ implies that $d(y,z)\leq K$ for some $z\in\widetilde{Z}$, so the
triangle inequality gives that $z\in B(f(u),2K)$. Now again apply the definition of co-coarse continuity
of $f$ to the point $u$, there exists a constant $\widetilde{K}>0$ depending only on $K$ such that
$$B(f(u),2K)\subset f(B(u,\widetilde{K}))^K.$$
Thus there exists $v\in B(u,\widetilde{K})$ such that $d(z,f(v))\leq K$. Then clearly we have
$d(y,f(v))\leq4K$ with $v\in B(x,\widetilde{K}+\delta)\cap f^{-1}(\widetilde{Z}^K)$. So we have shown that
$$B(f(x),\epsilon)\cap\widetilde{Z}^K\subset f(B(x,\widetilde{K}+\delta)\cap f^{-1}(\widetilde{Z}^K))^{4K}.$$
This implies that the mapping $f|_S$ from $f^{-1}(\widetilde{Z}^K)$ to $\widetilde{Z}^K$ is co-coarsely continuous.
Moreover, as a restriction mapping it inherits the property of coarse continuity, so it is
a coarse quotient mapping.

In addition, it is easy to see that $\widetilde{Z}^K$ is coarsely equivalent to $\widetilde{Z}$, and the
coarse equivalence is given by the selection mapping $p:\widetilde{Z}^K\rightarrow\widetilde{Z}$,
$p(a)=z_a$, where $z_a$ is any point in $\widetilde{Z}$ within distance $K$ from $a\in\widetilde{Z}^K$.
Then the composition $T\circ p\circ f|_S$ is a coarse quotient mapping from $S$ to $Z$, a contradiction.
\end{proof}

Recall that two metric spaces $X$ and $Y$ are said to be Lipschitz equivalent if there exists
a one-to-one Lipschitz mapping $f$ from $X$ onto $Y$ whose inverse is also Lipschitz. In this
case $f$ is called a Lipschitz homeomorphism. If $X$ is Lipschitz equivalent to a
subset of $Y$, we say that $X$ Lipschitz embeds into $Y$.

\begin{corollary}\label{metriclip}
Let $X$, $Y$ and $Z$ be metric spaces. Assume that there is a coarse quotient mapping from
$X$ to $Y$ that is Lipschitz for large distances and there is no coarse quotient mapping
from any subset of $X$ to $Z$ that is Lipschitz for large distances.
Then $Z$ does not Lipschitz embed into $Y$.
\end{corollary}

\begin{proof}
Assume that $\widetilde{Z}$ is a subset of $Y$ that is Lipschitz equivalent to $Z$ and
the Lipschitz homeomorphism is given by $T: \widetilde{Z}\rightarrow Z$. Let $f:X\rightarrow Y$
be a coarse quotient mapping that is Lipschitz for large distances. Other notations are as in
the proof of Theorem \ref{metriccontain}. Note that the composition $T\circ p\circ f|_S$ is a
coarse quotient mapping from $S$ to $Z$ that is Lipschitz for large distances, so again
we get a contradiction.
\end{proof}

\section{Coarse quotient mappings and Ultrapowers}

Using a tool called the uniform approximation by affine property
(UAAP), Bates, Johnson, Lindenstrauss, Preiss and Schechtman showed
in \cite{Bates} the following:

\begin{theorem}\label{Bates}
$X$ and $Y$ are two Banach spaces. Assume that $X$ is
super-reflexive and $Y$ is a uniform quotient of $X$. Then $Y^*$ is
crudely finitely representable in $X^*$. Consequently, $Y$ is
isomorphic to a linear quotient of some ultrapower of $X$.
\end{theorem}

One of the consequence of Theorem \ref{Bates} is that every uniform
quotient of $L_p$, $1<p<\infty$, is isomorphic to a linear quotient
of $L_p$. In this section, we will show that similar results hold
for coarse quotient mappings. The following lemma plays a key role;
it says that up to a constant, co-coarsely continuous mappings are
co-Lipschitz for large distances provided the target space is
metrically convex.

\begin{lemma}\label{core}
Let $X$ and $Y$ be two metric spaces and $f:X\rightarrow Y$ be a
co-coarsely continuous mapping with constant $K$. Assume that $Y$ is metrically convex.
Then for every $\epsilon>2K$, there exists $C=C(\epsilon)>0$ so that for all
$x\in X$ and $r\geq\epsilon$,
\begin{align}
\label{colip}B(f(x),r)\subset f(B(x,\frac{r}{C}))^K.
\end{align}
\end{lemma}

\begin{proof}
By the definition of co-coarsely continuous mapping, for
$\epsilon>2K$ there exists $\delta=\delta(\epsilon)>0$ so that
$B(f(x),\epsilon)\subset f(B(x,\delta))^K$ for all $x\in X$. We
claim that given $n\in\mathbb{N}$, for every $x\in X$ one has
\begin{align}
\label{n}B(f(x),n\epsilon)\subset f(B(x,2n\delta))^K.
\end{align}
Indeed, let $y\in B(f(x),n\epsilon)$. By the metric convexity of
$Y$, we can find $\{u_i\}_{i=0}^{2n}$ such that $u_0=f(x)$,
$u_{2n}=y$ and $d(u_i,u_{i-1})\leq\epsilon/2$ for all $i$. Put
$z_0=x$. The definition of co-coarsely continuous mapping yields
$z_1\in B(x,\delta)$ with $d(u_1,f(z_1))\leq K$, so triangle
inequality gives $d(u_2,f(z_1))\leq K+\epsilon/2<\epsilon$. Again by
definition there exists $z_2\in B(z_1,\delta)$ with
$d(u_2,f(z_2))\leq K$, and hence $d(u_3,f(z_2))\leq
K+\epsilon/2<\epsilon$...This process gives $\{z_i\}_{i=1}^{2n}$
inductively such that $d(z_i,z_{i-1})\leq\delta$ and
$d(u_i,f(z_i))\leq K$ for all $i$. Therefore $d(y,f(z_{2n}))\leq K$
and $z_{2n}\in B(x,2n\delta)$.

Now we claim that $C$ can be chosen as
$\epsilon/4\delta$. Indeed, for $r\geq\epsilon$, let
$k\in\mathbb{N}$ satisfy $k-1\leq r/\epsilon<k$. Note that $k\geq2$,
so by \eqref{n},
$$B(f(x),r)\subset B(f(x),k\epsilon)\subset
f(B(x,2k\delta))^K\subset f(B(x,\frac{r}{C}))^K.$$
\end{proof}

The next theorem states that in the Banach space setting one can
pass from coarse quotients to Lipschitz quotients by taking ultrapowers.
We refer to \cite{up} for ultraproduct techniques in Banach space theory.

\begin{theorem}\label{ultra}
Let $X$ and $Y$ be Banach spaces and $\mathcal{U}$ be a free ultrafilter on the natural numbers
$\mathbb{N}$. If $Y$ is a coarse quotient of $X$, then $Y_{\mathcal {U}}$ is a Lipschitz quotient of
$X_{\mathcal {U}}$.
\end{theorem}

\begin{proof}
Let $f$ be a coarse quotient mapping from $X$ to $Y$ with constant
$K$. In view of Lemma \ref{Johnson} we may assume that $K=0$.
By Proposition \ref{LipLarge} and Lemma \ref{core},
there are constants $L>0$ and $C>0$ such that
for all $x\in X$ and $r\geq1$,
\begin{align}
&f(B(x,r))\subset B(f(x),Lr),\nonumber\\
&B(f(x),r)\subset f(B(x,\frac{r}{C})).\nonumber
\end{align}
For each $n\in\mathbb{N}$, define $f_n: X\rightarrow Y$ by
$f_n(x)=f(nx)/n$. Then for all $x\in X$ and $r\geq1/n$,
\begin{align}
&f_n(B(x,r))\subset B(f_n(x),Lr),\nonumber\\
&B(f_n(x),r)\subset f_n(B(x,\frac{r}{C})).\nonumber
\end{align}
Define $T: X_\mathcal {U}\rightarrow Y_\mathcal {U}$ by
$T((x_n)_\mathcal {U})=(f_n(x_n))_\mathcal {U}$ for
$\tilde{x}=(x_n)_\mathcal {U}\in X_{\mathcal {U}}$.
Then it follows easily that for each $\tilde{x}\in X_\mathcal {U}$ and $r>0$,
\begin{align}
&T(B(\tilde{x},r))\subset B(T\tilde{x},Lr),\nonumber\\
&B(T\tilde{x},r)\subset T(B(\tilde{x},\frac{r}{C})).\nonumber
\end{align}
Therefore $T$ is a Lipschitz quotient mapping from $X_{\mathcal {U}}$ onto $Y_{\mathcal {U}}$.
\end{proof}

Recall that a Banach space $X$ is said to be crudely finitely
representable in a Banach space $Y$ if there exists
$1<\lambda<\infty$ so that for any finite-dimensional subspace
$E\subset X$, there exists a finite-dimensional subspace $F\subset
Y$ such that $d_{BM}(E,F)<\lambda$, where $d_{BM}$ is the
Banach-Mazur distance defined by
$$d_{BM}(E,F)=\inf\{\|T\|\|T^{-1}\|:T:E\rightarrow F~\text{is an isomorphism}\}.$$
If this is true for every $\lambda>1$, $X$ is said to be finitely
representable in $Y$.

\begin{theorem}\label{CQ FR}
$X$ and $Y$ are two Banach spaces. Assume that $X$ is
super-reflexive and $Y$ is a coarse quotient of $X$. Then $Y^*$ is
crudely finitely representable in $X^*$. Consequently, $Y$ is
isomorphic to a linear quotient of some ultrapower of $X$.
\end{theorem}

\begin{proof}
Let $\mathcal {U}$ be a free ultrafilter on $\mathbb{N}$. By Theorem
\ref{ultra} $Y_{\mathcal {U}}$ is a Lipschitz quotient of
$X_{\mathcal {U}}$. Note that $X_{\mathcal {U}}$ is super-reflexive,
so applying Theorem \ref{Bates} we get $(Y^*)_{\mathcal
{U}}=(Y_{\mathcal {U}})^*$ is crudely finitely representable in
$(X_{\mathcal {U}})^*=(X^*)_{\mathcal {U}}$. Also note that $Y^*$
can be viewed as a subspace of $(Y^*)_{\mathcal {U}}$ and
$(X^*)_{\mathcal {U}}$ is finitely representable in $X^*$, so we are
done.
\end{proof}

\begin{corollary}\label{CQ Hilbert}
A Banach space that is a coarse quotient of a Hilbert space must be
isomorphic to a Hilbert space.
\end{corollary}

\begin{proof}
Let $Y$ be a Banach space that is a coarse quotient of a Hilbert
space. By Theorem \ref{CQ FR}, $Y^*$ is crudely finitely
representable in a Hilbert space, hence must be isomorphic to a
Hilbert space. Therefore $Y$ is isomorphic to a Hilbert space.
\end{proof}

\begin{corollary}\label{CQ Lp}
If a Banach space $Y$ is a coarse quotient of $L_p$, $1<p<\infty$,
then $Y$ is isomorphic to a linear quotient of $L_p$.
\end{corollary}

\begin{proof}
Note that $Y$ must be super-reflexive and separable, so $Y^*$ is also
separable. Moreover, by Theorem \ref{CQ FR} $Y^*$ is crudely finitely representable
in $L_q$ with $1/p+1/q=1$. It follows that $Y^*$ isomorphically embeds into $L_q$ (see
\cite{LP}), i.e., $Y$ is isomorphic to a linear quotient of $L_p$.
\end{proof}

\section{Coarse quotient mappings and property ($\beta$)}

In this section our goal is to show that $\ell_q$ is not a coarse
quotient of $\ell_p$ for $1<p<q<\infty$. The idea comes from Lima
and Randriantoanina \cite{LimaLova}, where they proved the result in
the uniform category. Their proof relies on a geometric property
introduced by Rolewicz in \cite{beta} which is now called property
($\beta$). Let us recall the equivalent definition of property
($\beta$) given by Kutzarova \cite{Kuza}.

\begin{definition}
A Banach space $X$ is said to have property ($\beta$) if for any
$\epsilon>0$ there exists $0<\delta<1$ such that for every element
$x\in B_X$ and every sequence $(x_n)_{n=1}^{\infty}\subset B_X$ with
$\text{sep}(\{x_n\})\geq\epsilon$ there exists an index $i$ such
that
$$\|\dfrac{x-x_i}{2}\|\leq1-\delta.$$
Here the separation of the sequence is defined by
$$\text{sep}(\{x_n\})=\inf\{\|x_n-x_m\|:n\neq m\}$$
and $B_X$ is the closed unit ball of $X$.
\end{definition}

Ayerbe, Dom\'{\i}nguez Benavides and Cutillas defined a modulus for
the property ($\beta$) in \cite{betaMod}, as follows:

\vspace{2mm}

$\bar{\beta}_X:[0,a]\rightarrow[0,1]:$
$$\bar{\beta}_X(\epsilon)=1-\sup\left\{\inf\left
\{\frac{\|x-x_n\|}{2}:n\in\mathbb{N}\right\}:x\in B_X, (x_n)_{n=1}^{\infty}\subset B_X,
\text{sep}(\{x_n\})\geq\epsilon\right\},$$
where the constant $a\in[1,2]$ is the maximal separation of sequences in $B_X$, which
depends on the Banach space $X$. We say that the ($\beta$)-modulus has
power type $p$ if there exists a constant $A>0$ such that $\bar{\beta}_X(\epsilon)\geq
A\epsilon^p$ for all $\epsilon\in[0,a]$. The ($\beta$)-modulus of $\ell_p$ ($1<p<\infty$) was
also computed in \cite{betaMod} and is given by

$$\bar{\beta}_{\ell_p}(\epsilon)=1-\frac{1}{2}\left(\left(1+\left(1-
\frac{\epsilon^p}{2}\right)^{1/p}\right)^p+\frac{\epsilon^p}{2}\right)^{1/p},~
0\leq\epsilon\leq2^{1/p}.$$
One can directly check that $\bar{\beta}_{\ell_p}$ has
power type $p$ (see \cite{LimaLova}). Indeed, it was shown in \cite{UQ4}
that the ($\beta$)-modulus of any $\ell_p$-sum of finite-dimensional spaces
is the same as that of $\ell_p$, hence has power type $p$.

On the other hand, we still need to make use of the co-Lipschitz for
large distances principle. Let $K$ be as in the Definition \ref{CQ}.
For $d>2K$, denote $C_d$ the supremum of all $C$ that satisfies
\eqref{colip}. Clearly, $C_d$ is nondecreasing with respect to $d$.
Moreover, we have:

\begin{lemma}\label{lim}
Let $X$ and $Y$ be two metric spaces and $f:X\rightarrow Y$ be a
coarse quotient mapping with constant $K$. Assume that $Y$ is metrically
convex and $f$ is Lipschitz for large distances.
Then $\lim\limits_{d\rightarrow\infty}C_d<\infty$.
\end{lemma}

\begin{proof}
It suffices to show that $\{C_d\}_{d>2K}$ is bounded.
Let $l(1)>0$ be the Lipschitz constant of $f$ when distance
of points in $X$ is at least 1. For any $d>2K$,
by Lemma \ref{core} there exists $C=C(d)>0$ such
that for all $x\in X$ and $r\geq d$, $B(f(x),r)\subset
f(B(x,r/C))^K$. Note that
$$f(B(x,\dfrac{r}{C}))^K\subset
B(f(x),\omega_f(\dfrac{r}{C}))^K=B(f(x),\omega_f(\dfrac{r}{C})+K),$$
so
$$r\leq\omega_f(\frac{r}{C})+K\leq\max\{\omega_f(1),
l(1)\cdot\dfrac{r}{C}\}+K.$$
Now pick $r$ large so that $r>C$. Then $r\leq\dfrac{\eta r}{C}+K$,
where $\eta:=\max\{\omega_f(1), l(1)\}$.
It follows that $C\leq\dfrac{\eta r}{r-K}<2\eta$,
and hence $C_d\leq2\eta$.

\end{proof}

\begin{theorem}\label{main}
Let $X$ and $Y$ be Banach spaces and $1<p<q<\infty$.
Assume that the ($\beta$)-modulus of $X$ has
power type $p$ and $Y$ contains a subspace isomorphic to $\ell_q$.
Then there is no coarse quotient mapping from any subset of $X$
to $Y$ that is Lipschitz for large distances.
\end{theorem}

\begin{proof}
In view of Corollary \ref{metriclip}, we may without loss
of generality assume that $Y=\ell_q$. Suppose that there exist a subset $S$ of $X$
and a coarse quotient mapping $f:S\rightarrow\ell_q$ with constant $K$
so that $f$ is Lipschitz for large distances.
Let $C$ be the finite limit given by Lemma \ref{lim} with respect to $f$.
Fix a small $0<\epsilon<1$ and consider large
$d_0$ so that $d_0/3>2K$ and $C-\epsilon<C_{d_0/3}\leq
C_{d_0}\leq C<C+\epsilon$. Since $C_{d_0}<C+\epsilon$, by the
definition of $C_{d_0}$ as a supremum, there exist $z_{\epsilon}\in
S$ and $R\geq d_0$ such that
$$B(f(z_{\epsilon}), R)\nsubseteq
f(B(z_{\epsilon},\dfrac{R}{C+\epsilon})\cap S)^K,$$
so there exists $y_{\epsilon}\in\ell_q$ satisfying
$0<\|y_{\epsilon}-f(z_{\epsilon})\|:=\gamma\leq R$ and
\begin{align}\label{empty}
B(y_{\epsilon},K)\cap f(B(z_{\epsilon},\dfrac{R}{C+\epsilon})\cap
S)=\emptyset.
\end{align}

Let $m$ and $M$ be two points on the line segment with endpoints
$y_{\epsilon}$ and $f(z_{\epsilon})$ such that
$\|y_{\epsilon}-M\|=\|M-m\|=\|m-f(z_{\epsilon})\|=\gamma/3$. Since
$C-\epsilon<C_{d_0/3}$ and $R/3\geq d_0/3$, by the definition of
$C_{d_0/3}$ as a supremum we have
$$B(f(z_{\epsilon}),\dfrac{R}{3})\subset
f(B(z_{\epsilon},\dfrac{R}{3(C-\epsilon)})\cap S)^K.$$
Note that $\|m-f(z_{\epsilon})\|=\gamma/3\leq
R/3$, so there exists $x\in S$ satisfying
$\|x-z_{\epsilon}\|\leq\dfrac{R}{3(C-\epsilon)}$ and $\|m-f(x)\|\leq
K$.

Let $(e_n)_{n=1}^{\infty}$ be the unit vector basis for $\ell_q$,
and denote by $(M-m)_N$ the truncation of $M-m$ supported on the first
$N$ coordinates. Choose $N\in\mathbb{N}$ large so that
$\|(M-m)-(M-m)_N\|_q<\epsilon \dfrac{R}{3}$. For $n>N$, set
\begin{align}\label{yn}
y_n:=\epsilon^{1/q}\dfrac{R}{3}e_n+(1-\epsilon)^{1/q}(M-m)_N+m.
\end{align}
Then
\begin{align}
\|y_n-m\|_q^q&=\|\epsilon^{1/q}\dfrac{R}{3}e_n+(1-\epsilon)^{1/q}(M-m)_N\|_q^q\nonumber\\
&=\epsilon(\dfrac{R}{3})^q+(1-\epsilon)\|(M-m)_N\|_q^q\leq\epsilon(\dfrac{R}{3})^q+(1-\epsilon)(\dfrac{\gamma}{3})^q
\leq(\dfrac{R}{3})^q,\nonumber
\end{align}
so $\|y_n-m\|\leq R/3$. Choose $d_0$ large enough so that
$d_0\epsilon\geq K$, we then have
$$\|y_n-f(x)\|\leq\|y_n-m\|+\|m-f(x)\|\leq
\dfrac{R}{3}+K\leq(\dfrac{1}{3}+\epsilon)R.$$
Noting that $(\dfrac{1}{3}+\epsilon)R\geq\dfrac{d_0}{3}$,
again by the definition of $C_{d_0/3}$,
\begin{align}\label{zn}
B(f(x),(\dfrac{1}{3}+\epsilon)R)\subset
f(B(x,\dfrac{(\frac{1}{3}+\epsilon)R}{C-\epsilon})\cap S)^K,
\end{align}
so there exists $z_n\in S$ satisfying
$\|z_n-x\|\leq\dfrac{(\frac{1}{3}+\epsilon)R}{C-\epsilon}$ and
$\|y_n-f(z_n)\|\leq K$. Now we estimate $\|y_n-y_{\epsilon}\|$:
\begin{align}
&\|y_n-y_{\epsilon}\|_q^q\nonumber\\
&=\left\|\epsilon^{1/q}\dfrac{R}{3}e_n+(1-\epsilon)^{1/q}(M-m)_N-(y_{\epsilon}-m)\right\|_q^q\nonumber\\
&=\left\|\epsilon^{1/q}\dfrac{R}{3}e_n+(1-\epsilon)^{1/q}(M-m)_N-2(M-m)\right\|_q^q\nonumber\\
&=\bigg\|\epsilon^{1/q}\dfrac{R}{3}e_n+2\Big((M-m)_N-(M-m)\Big)+(1-\epsilon)^{1/q}(M-m)_N-2(M-m)_N\bigg\|_q^q\nonumber\\
&=\left\|\epsilon^{1/q}\dfrac{R}{3}e_n+2\Big((M-m)_N-(M-m)\Big)\right\|_q^q+\left(2-(1-\epsilon)^{1/q}\right)^q\|(M-m)_N\|_q^q\nonumber\\
&\leq\left(\epsilon^{1/q}\dfrac{R}{3}+2\epsilon\dfrac{R}{3}\right)^q+\big(2-(1-\epsilon)\big)^q\left(\dfrac{R}{3}\right)^q\nonumber\\
&\leq3^q\epsilon\left(\dfrac{R}{3}\right)^q+(1+\epsilon)^q\left(\dfrac{R}{3}\right)^q\nonumber\\
&<(1+2\cdot3^q\epsilon)\left(\dfrac{R}{3}\right)^q\nonumber,
\end{align}
so
$$\|y_n-y_{\epsilon}\|\leq\left(1+2\cdot3^q\epsilon\right)^{1/q}\dfrac{R}{3}\leq\left(1+\dfrac{2\cdot3^q\epsilon}{q}\right)\dfrac{R}{3},$$
and hence
\begin{align}
\|y_{\epsilon}-f(z_n)\|&\leq\|y_{\epsilon}-y_n\|+\|y_n-f(z_n)\|\nonumber\\
&\leq\left(1+\dfrac{2\cdot3^q\epsilon}{q}\right)\dfrac{R}{3}+K\nonumber\\
&\leq\left(1+\dfrac{2\cdot3^q\epsilon}{q}\right)\dfrac{R}{3}+\epsilon
R=\left(\dfrac{1}{3}+\epsilon+\dfrac{2\cdot3^{q-1}\epsilon}{q}\right)R:=\rho_\epsilon
R.\nonumber
\end{align}
Noting that $\rho_\epsilon R>R/3\geq d_0/3$, again by the definition
of $C_{d_0/3}$ we have
\begin{align}\label{xn}
B(f(z_n), \rho_\epsilon R)\subset
f(B(z_n,\dfrac{\rho_\epsilon R}{C-\epsilon})\cap S)^K,
\end{align}
so there exists $x_n\in S$ satisfying
$\|x_n-z_n\|\leq\dfrac{\rho_\epsilon R}{C-\epsilon}$ and
$\|y_\epsilon-f(x_n)\|\leq K$. In view of \eqref{empty}, we have
$\|x_n-z_\epsilon\|>\dfrac{R}{C+\epsilon}$. Also note that
$\rho_\epsilon\downarrow\dfrac{1}{3}$ as $\epsilon\downarrow0$, so
if $\epsilon$ is chosen small enough so that
$\dfrac{1}{C+\epsilon}-\dfrac{\rho_\epsilon}{C-\epsilon}>0$, then
the triangle inequality gives
$$\|z_\epsilon-z_n\|\geq\|z_\epsilon-x_n\|-\|x_n-z_n\|\geq\dfrac{R}{C+\epsilon}-\dfrac{\rho_\epsilon
R}{C-\epsilon}>0.$$

On the other hand, we could choose large $d_0$ so that
$(2\epsilon)^{1/q}\cdot\dfrac{d_0}{6}>\omega_f(1)+2K$. Then for
$k,n>N$ with $k\neq n$,
\begin{align}
\omega_f(1)+2K<(2\epsilon)^{1/q}\cdot\dfrac{R}{3}&=\|y_n-y_k\|_q\nonumber\\
&\leq\|y_n-f(z_n)\|+\|f(z_n)-f(z_k)\|+\|y_k-f(z_k)\|\nonumber\\
&\leq2K+\omega_f(\|z_n-z_k\|).\nonumber
\end{align}
Thus $\omega_f(\|z_n-z_k\|)>\omega_f(1)$ and it follows that
$\|z_n-z_k\|>1$ since $\omega_f$ is nondecreasing. Hence the
Lipschitz for large distances property gives
$$(2\epsilon)^{1/q}\cdot\dfrac{R}{3}=\|y_n-y_k\|_q\leq l(1)\|z_n-z_k\|+2K\leq l(1)\|z_n-z_k\|+(2\epsilon)^{1/q}\cdot\dfrac{R}{6},$$
where $l(\cdot)$ is given in Proposition \ref{LipLarge}. This
implies that $\|z_n-z_k\|\geq(2\epsilon)^{1/q}\cdot\dfrac{R}{6l(1)}$.

In summary, for all $n,k>N$ with $n\neq k$ we have:

$$\|z_n-z_k\|\geq(2\epsilon)^{1/q}\cdot\dfrac{R}{6l(1)},\hspace{5mm}\|z_\epsilon-z_n\|\geq\dfrac{R}{C+\epsilon}-\dfrac{\rho_\epsilon
R}{C-\epsilon},$$
$$\|z_{\epsilon}-x\|\leq\dfrac{R}{3(C-\epsilon)},\hspace{5mm}\|z_n-x\|\leq\dfrac{(\frac{1}{3}+\epsilon)R}{C-\epsilon}.$$
Assume that $\epsilon$ is small enough, by the definition of
$\bar{\beta}_X$ we get
\begin{align}\label{beta}
\bar{\beta}_X\left(\dfrac{(2\epsilon)^{1/q}}{6l(1)}\cdot\dfrac{C-\epsilon}{\frac{1}{3}+\epsilon}\right)\leq
1-\frac{1}{2}\cdot\dfrac{\dfrac{1}{C+\epsilon}-\dfrac{\rho_\epsilon}{C-\epsilon}}{\dfrac{\frac{1}{3}+\epsilon}{C-\epsilon}}.
\end{align}
Note that $\bar{\beta}_X(\cdot)$ is nondecreasing and
has power type $p$, so if we started with
small $\epsilon$ so that
$\dfrac{C-\epsilon}{1+3\epsilon}>\dfrac{C}{2}$,
$$\text{left
side of \eqref{beta}}
\geq\bar{\beta}_X\left(\dfrac{C}{2^{2-\frac{1}{q}}l(1)}\cdot\epsilon^{1/q}\right)\geq
A\epsilon^{p/q}$$ for some $A>0$, whereas
\begin{align}
\text{right side of \eqref{beta}}
&\leq1-(\dfrac{1}{2}-\dfrac{\epsilon}{C})\cdot\dfrac{3}{1+3\epsilon}+\dfrac{1}{2}+\dfrac{3^{q-1}\epsilon}{(\frac{1}{3}+\epsilon)q}\nonumber\\
&\leq\frac{3}{2}\left(1-\dfrac{1}{1+3\epsilon}\right)+\dfrac{3\epsilon}{C}+\dfrac{3^{q}\epsilon}{q}\nonumber\\
&\leq\left(\frac{9}{2}+\dfrac{3}{C}+\dfrac{3^{q}}{q}\right)\epsilon\nonumber,
\end{align}
so
$$\left(\frac{9}{2}+\dfrac{3}{C}+\dfrac{3^{q}}{q}\right)\epsilon^{1-\frac{p}{q}}\geq A.$$
Since $1<p<q<\infty$, we get a contradiction by letting
$\epsilon\rightarrow0$.

\end{proof}

\begin{remark}
In the case when $S$ is a subspace of $X$, Lemma \ref{Johnson} can be applied to simplify the proof.
\end{remark}

\begin{corollary}\label{lplq}
Let $Y$ be a Banach space and $1<p<q<\infty$. If $\ell_q$ isomorphically embeds into $Y$,
then $Y$ is not a coarse quotient of any $\ell_p$-sum of finite-dimensional spaces.
In particular, $\ell_q$ is not a coarse quotient of $\ell_p$.
\end{corollary}

%\begin{remark}
%Since the modulus of property ($\beta$) of any $\ell_p$-sum of
%finite-dimensional spaces is the same as that of $\ell_p$ (see
%\cite{UQ4}), so Theorem \ref{lplq} still holds if $\ell_p$ is
%replaced by any $\ell_p$-sum of finite-dimensional spaces.
%\end{remark}

%Note that the modulus of property ($\beta$) of any $\ell_p$-sum of
%finite-dimensional spaces is the same as that of $\ell_p$ (see
%\cite{UQ4}), so Theorem \ref{lplq} can be generalized in the
%following way.

%\begin{corollary}
%$1<p<q<\infty$. If a Banach space $Y$ is a coarse quotient of a
%convex subset of an $\ell_p$-sum of finite-dimensional spaces, then
%$Y$ does not contain any subspace isomorphic to $\ell_q$.
%\end{corollary}

A consequence of Corollary \ref{lplq} is the following isomorphic
characterization of coarse quotients of $\ell_p$ for $1<p<2$.

\begin{corollary}
If a Banach space $Y$ is a coarse quotient of $\ell_p$, $1<p<2$,
then $Y$ is isomorphic to a linear quotient of $\ell_p$.
\end{corollary}

\begin{proof}
A similar proof as in Corollary \ref{CQ Lp} shows that $Y$ is
isomorphic to a linear quotient of $L_p$. On the other hand, by
Corollary \ref{lplq}, $\ell_2$ is not isomorphic to a linear
quotient of $Y$. Therefore by a result of Johnson and Odell \cite{JO}
$Y$ must be isomorphic to a linear quotient of $\ell_p$.
\end{proof}

It follows directly from Theorem \ref{CQ FR} that $c_0$ is not a
coarse quotient of any super-reflexive space, but it was shown in
\cite{sbeta1} and \cite{sbeta2} that there are Banach spaces with
property ($\beta$) that are not super-reflexive. Indeed, a slight
modification to the computation used for Theorem \ref{main} gives
the following result.

\begin{theorem}
Let $X$ and $Y$ be Banach spaces. Assume that $X$ has property ($\beta$)
and $Y$ contains a subspace isomorphic to $c_0$.
Then there is no coarse quotient mapping from any subset of $X$ to
$Y$ that is Lipschitz for large distances.
\end{theorem}

\begin{proof}
In view of Corollary \ref{metriclip}, we may without loss of generality assume that $Y=c_0$.
Suppose that there exist a subset $S$ of $X$
and a coarse quotient mapping $f:S\rightarrow c_0$ with constant $K$
so that $f$ is Lipschitz for large distances.
Let $C$ be the finite limit given by Lemma \ref{lim} with respect to $f$.
Now we proceed the proof for Theorem \ref{main} until \eqref{yn}, the choice of $y_n$.
Instead, for $n>N$, set
$$y_n:=\dfrac{R}{3}e_n+(M-m)_N+m.$$ Then
$$\|y_n-m\|_{\infty}=\|\dfrac{R}{3}e_n+(M-m)_N\|_{\infty}=\max\left\{\dfrac{R}{3},
\|(M-m)_N\|_{\infty}\right\}=\dfrac{R}{3},$$
so as in Theorem \ref{main} we can choose $z_n$ by \eqref{zn}. Moreover,
\begin{align}
\|y_n-y_{\epsilon}\|_{\infty}
&=\left\|\dfrac{R}{3}e_n+(M-m)_N-(y_{\epsilon}-m)\right\|_{\infty}\nonumber\\
&=\left\|\dfrac{R}{3}e_n+(M-m)_N-2(M-m)\right\|_{\infty}\nonumber\\
&=\bigg\|\dfrac{R}{3}e_n+2\Big((M-m)_N-(M-m)\Big)-(M-m)_N\bigg\|_{\infty}\nonumber\\
&=\max\left\{\bigg\|\dfrac{R}{3}e_n+2\Big((M-m)_N-(M-m)\Big)\bigg\|_{\infty}, \|(M-m)_N\|_{\infty}\right\}\nonumber\\
&\leq\max\left\{\dfrac{R}{3}+\dfrac{2\epsilon R}{3},
\dfrac{R}{3}\right\}=(1+2\epsilon)\dfrac{R}{3}.\nonumber
\end{align}
Thus
$$\|y_{\epsilon}-f(z_n)\|\leq\|y_{\epsilon}-y_n\|+\|y_n-f(z_n)\|\leq(1+2\epsilon)\dfrac{R}{3}+K\leq(\dfrac{1}{3}+\dfrac{5\epsilon}{3})R
:=\tilde{\rho}_\epsilon R.\nonumber$$ Note that
$\tilde{\rho}_\epsilon R>R/3\geq d_0/3$, so similarly $x_n$ can be chosen by \eqref{xn}.

On the other hand, we could choose large $d_0$ so that
$d_0/6>\omega_f(1)+2K$. Then for $k,n>N$ with $k\neq n$,
\begin{align}
\omega_f(1)+2K<\dfrac{R}{3}&=\|y_n-y_k\|_{\infty}\leq2K+\omega_f(\|z_n-z_k\|),\nonumber
\end{align}
which again implies that $\|z_n-z_k\|>1$, and hence the Lipschitz for large
distances property gives
$$\dfrac{R}{3}=\|y_n-y_k\|_{\infty}\leq l(1)\|z_n-z_k\|+2K\leq l(1)\|z_n-z_k\|+\dfrac{R}{6},$$
where $l(\cdot)$ is given in Proposition \ref{LipLarge}. This
implies that $\|z_n-z_k\|\geq\dfrac{R}{6l(1)}$.

In summary, for all $n,k>N$ with $n\neq k$ we have:

$$\|z_n-z_k\|\geq\dfrac{R}{6l(1)},\hspace{5mm}\|z_\epsilon-z_n\|\geq\dfrac{R}{C+\epsilon}-\dfrac{\tilde{\rho}_\epsilon
R}{C-\epsilon},$$
$$\|z_{\epsilon}-x\|\leq\dfrac{R}{3(C-\epsilon)},\hspace{5mm}\|z_n-x\|\leq\dfrac{(\frac{1}{3}+\epsilon)R}{C-\epsilon}.$$
Note that if $\epsilon$ is chosen small enough so that $\epsilon\leq
C/2$, one has
$$\Bigg\|\dfrac{(z_n-x)-(z_k-x)}{\dfrac{(\frac{1}{3}+\epsilon)R}{C-\epsilon}}\Bigg\|\geq
\dfrac{1}{6l(1)}\cdot\dfrac{C-\epsilon}{\frac{1}{3}+\epsilon}\geq\dfrac{C}{16l(1)}>0.$$
Therefore by the definition of property ($\beta$) there exists $0<\delta<1$ independent of $\epsilon$
and an index $i>N$ such that
$$\|z_{\epsilon}-z_i\|\leq\dfrac{(\frac{1}{3}+\epsilon)R}{C-\epsilon}\cdot(2-2\delta).$$
It follows that
$$\dfrac{R}{C+\epsilon}-\dfrac{\tilde{\rho}_\epsilon
R}{C-\epsilon}\leq\dfrac{(\frac{1}{3}+\epsilon)R}{C-\epsilon}\cdot(2-2\delta).$$
Let $\epsilon\rightarrow0$ we get $2\leq2-2\delta$, a contradiction.
\end{proof}

\begin{corollary}
$c_0$ is not a coarse quotient of any Banach space with property ($\beta$).
\end{corollary}

{\bf Acknowledgement.} The author would like to thank William B.
Johnson for his guidance and many helpful suggestions which improved
the paper.

\bibliographystyle{amsplain}

\end{document}